\documentclass[12pt]{amsart}

\usepackage{amscd}
\usepackage{verbatim}

\advance\hoffset by -0.5 truecm

\usepackage{amssymb}
\usepackage{subfigure}
\newtheorem{Theorem}{Theorem}[section]

\newtheorem{Lemma}[Theorem]{Lemma}

\newtheorem{Corollary}[Theorem]{Corollary}

\advance\hoffset by -0.5 truecm

\usepackage{subcaption}
\usepackage{amssymb}
\usepackage{graphicx}

\def\V{\mbox{Var}}

\def\R\re

\def\V{\bf V}

\def \re{{\mathbb R}}

\def \0{\lambda_{0}}

\begin{document}
\hyphenation{Ya-ma-be co-rres-pon-ding hy-po-the-sis iso-pe-ri-me-tric gen-er-at-ed man-i-folds me-tric}
\title{Lower bounds for  isoperimetric profiles and Yamabe constants}

%Affilation

%\author{ Author Name \thanks{Affiliation} }]

%Institution, (department), city, (state), country. Email.

\author[J. M. Ruiz]{Juan Miguel Ruiz$^\dagger$}
 \thanks{$^\dagger$ ENES  UNAM. Departamento  de Matem\'aticas. Le\'on, Gto., M\'exico. mruiz@enes.unam.mx}

 %\address{Juan Miguel Ruiz, ENES UNAM \\
  %         37684 \\
   %       Le\'on. Gto. \\
    %      M\'exico.}
%\email{mruiz@enes.unam.mx}

\author[A. V. Juarez]{Areli V\'azquez Ju\'arez$^{\ddagger *}$}
\thanks{$^\ddagger$ ENES UNAM.  Departamento de Matem\'aticas. Le\'on, Gto., M\'exico. areli@enes.unam.mx.}
\thanks{* Corresponding author.}
%\address{Areli V\'azquez Juarez, ENES UNAM \\
%	37684 \\
%	Le\'on. Gto. \\
%	M\'exico.}
%\email{areli@enes.unam.mx}

\subjclass{ 53C21}

\keywords{Isoperimetric profile, Symmetrization,Yamabe constants}

\begin{abstract}  
	We estimate explicit lower bounds for the isoperimetric profiles of the Riemannian product of a compact manifold and the Euclidean space with the flat metric, $(M^m\times \re^n,g+g_E)$, $m,n>1$. In particular, we introduce a lower bound for the isoperimetric profile of  $M^m\times \re^n$ for regions of large volume and we improve on previous estimates of lower bounds for the isoperimetric profiles of $S^2 \times \re^2$, $S^3 \times \re^2$, $S^2 \times \re^3$. We also discuss some applications of these results in order to improve known lower bounds for the Yamabe invariant of certain manifolds.
	
	%, which are involved in formulas for the Yamabe invariant of products \cite{PetRu2,AFP} and the Yamabe invariants of manifolds obtained through surgery \cite{ADH1, ADH2}. We use techniques and ideas developed by F. Morgan \cite{Morgan}, A. Ros \cite{Ros} and M. Ritor\'e and   Vernadakis \cite{RiVer}

\end{abstract}

\maketitle

\section{Introduction}
\begin{comment}

\begin{Theorem}
Consider the product $S^2\times \re^n$, $n>1$. Isoperimetric regions of volume bigger than some $v_0>0$ are not of ball type. Moreover 
	$v_0$ is bounded from above by $nnnnn $.
\end{Theorem}
\end{comment}
%The isoperimetric problem is  classical  in differential geometry. 
 On  a Riemannian manifold $(M^n,g)$ of volume $Vol^n(M)$, 
we call a closed region $\Omega\subset M$   of volume $t$, $0<t<Vol^n(M)$, an isoperimetric region, if  it is such that its boundary area, $Vol^{n-1}(\partial \Omega)$, is minimal among all compact hypersurfaces $\Sigma \subset M$ enclosing a region of volume $t$.  The isoperimetric  profile of a manifold $(M,g)$ is the function $I_{(M,g)} : (0,Vol^n(M) ) \rightarrow (0,\infty )$ defined by the infimum,

$$I_{(M,g)} (t) = \inf \{ Vol^{n-1}( \partial \Omega ) : Vol^n (\Omega) = t , \Omega \subset M^n, \Omega \mathit{ \ a \ closed \  region} \}.$$

 \noindent Where $Vol^{n}(\Omega)$, the volume of a closed region $\Omega \subset M^n$,  denotes the $n-$dimensional Riemannian measure of $\Omega$. Meanwhile,  $Vol^{n-1}(\partial \Omega)$, the area of its boundary, denotes  the $(n-1)$-dimensional Riemannian measure of the boundary $\partial \Omega$. We remark that the volume of the manifold $M^n$ may be infinite. See \cite{Ros} for more details on the isoperimetric profile and  the isoperimetric   problem in general. %We will simply write $I_{M}$ when the metric $g$ is understood from the context.  
 
Even though the study of isoperimetric regions is a classical problem, the precise  isoperimetric profile is known for few manifolds. Examples include space forms ($\re^n$, $S^n$, $\mathbb{H}^n$, with the flat, round and hyperbolic metrics, respectively), cylinders  ($S^n\times \re$ with the product metric, see \cite{Pedro}) and low dimensional products of space forms with one dimensional circles ($S^n\times S^1$, $\re^n\times S^1$, $\mathbb{H}^n\times S^1$, $2\leq n\leq 7$, see \cite{PeRit}). See also some recent results on $\re \mathbb P^n$ and space lenses \cite{Viana}.  On the other hand,  seemingly  simple products like $S^2\times \re^2$ or $S^3\times \re^2$,  with the product metric, have resulted harder to understand than their factors and their explicit isoperimetric profiles are not known. Therefore, qualitative results in this direction, which include lower bounds for isoperimetric profiles of product manifolds, or characterizations of   isoperimetric regions for products,  have  been of interest, see for example, \cite{Morgan2}, \cite{Morgan3}, \cite{Pet}, \cite{PetRu2} and  \cite{RuizVaz}.

%It is known explicitly, for example, for the Euclidean space with the flat metric, $(\re^n,g_E)$, the $m$-sphere with the round metric $(S^n, g_0)$,  and the hyperbolic space $(\H^n,g_H)$.  It is also known for cylinders  $(S^n\times \re, g_0+dt^2)$ by the work of R. Pedrosa \cite{Pedro}, and for the Riemannian product of a low dimensional space form  with $S^1$, i.e., $(S^1\times \re^n,dt^2+g_E)$, $(S^1\times S^n, dt^2+g_0)$, $(S^1\times\H^n,dt^2+g_H)$ ($2 \leq n\leq 7$), by the work of R. Pedrosa and M. Ritor\'e \cite{PeRit}. See also the work of C. Viana on $\re \mathbb{P}^n$ and on lens spaces with large fundamental groups \cite{Viana, Viana2}. Nevertheless, 

%	The isoperimetric profile of $ S^m\times \re$ was studied by J. Petean in \cite{Pet}, while $ S^1\times \re^n$, $n\leq 7$, was shown explicitly by Pedrosa and Ritor\'e in \cite{PeRit}. 
In this article, we  estimate explicit lower bounds for the isoperimetric profile of the product of a compact manifold and Euclidean space with the flat metric $g_E$, $(M\times \re^n,g+g_E)$, using techniques and ideas developed by F. Morgan \cite{Morgan}, A. Ros \cite{Ros} and M. Ritor\'e and E.  Vernadakis \cite{RiVer}. In particular, we introduce a lower bound for the isoperimetric profile of  $(M\times \re^n,g+g_E)$ for regions of large volume.

Let $(M^m,g)$ be a closed (compact, without boundary) Riemannian manifold. Let $D^n_R$ be a disk of radius $R>0$ in $\re^n$. If the volume of the region	$M^m\times D^n_R\subset M^m\times \re^n$ is $v$, its area can be computed to be  given by the function

\begin{equation}
	\label{Fmn}
F_{M,n}(v)=Vol(M)^{\frac{1}{n}}\gamma_n v^{\frac{n-1}{n}}
\end{equation}
where $\gamma_n$ is the n-Euclidean isoperimetric constant. %and $V_m$ the volume of the round m-sphere $(S^m,g_0^m)$ of unit radius..
	 We will denote by $C_{M,n}$, the  constant coefficient of $F_{M,n}(v)$, this is, $C_{M,n} =Vol(M)^{\frac{1}{n}}\gamma_n$.
\begin{Theorem}
	\label{Thm1}
 Let $m,n>1$. Let $(M^m,g)$ be a compact Riemannian manifold without boundary and $h_m(v)$ a concave lower bound for its isoperimetric profile.   	Let $\alpha \in \re$, $0<\alpha<1$.Then, for $v>v_{0}$, 
	
	\begin{equation}
		\label{alpha}
	%	I_{S^m\times \re^n}(v)\geq \alpha^{2-\frac{1}{m}}C_{M,n} v^{\frac{n-1}{n}}\geq \alpha^{2} F_{M,n}(v)
\alpha^{2-\frac{1}{n}} F_{M,n}(v) \leq	I_{(M^m\times \re^n,g+g_E)}(v),%\leq  F_{M,n}(v)
	\end{equation}
	
\noindent	where  $v_0=\left(\frac{C_{M,n}}{k(1-\alpha)^2}\right)^n$ and  $k= \frac{\alpha Vol(M)}{h_m(\alpha Vol(M))}$.
\end{Theorem}

Note that, since the regions $M^m\times D^n_R$ are actual closed regions in  $M^m\times \re^n$, then
$$\alpha^{2} F_{M,n}(v) \leq	I_{(M^m\times \re^n,g+g_E)}(v)\leq  F_{M,n}(v).$$
This is, the closer to 1 one chooses $\alpha$ to be, the closer the bound is to the isoperimetric profile  (although $v_0$ will get bigger). It was actually proven by Vernadakis and Ritor\'e  in \cite{RiVer}, see also the work of  Gonzalo in \cite{Gonz}, that eventually,  for some   $v^*>0$, $ F_{M,n}(v) =	I_{(M^m\times \re^n,g+g_E)}(v)$ holds for $v>v^*$. Nevertheless,  said $v^*$ has not been quantified or bounded for general manifolds $M^m$ (see \cite{RuizVaz} for  work of the authors in  this direction, for the particular case of the product of a flat torus and Euclidean space, $T^m\times \re^n$, for small dimensions $m+n$).

\begin{comment}

Note that the left-hand side inequality is trivial, since $S^m\times D^n_R$ are actual regions in  $S^m\times \re^n$, then		$I_{S^m\times \re^n}(v)\leq   C_{M,n} v^{\frac{n-1}{n}}$. So that the closer one chooses $\alpha$ to 1, the closer the bound is to being optimal. It was actually proven by Vernadakis and Ritor\'e  in \cite{RiVer}, and by Gonzalo in \cite{Gonz}, that eventually,  after some big $v^*>0$, inequality (\ref{alpha}) holds for $v>v^*$ and $\alpha=1$ (of course, as argued before, $\alpha=1$  implies the equality in (\ref{alpha})). Nevertheless, those results are of a qualitative nature and said $v^*$ has not been quantified or bounded in any way before, to the best of  our knowledge. 
\end{comment}

By studying the isoperimetric profile of manifolds of the type $(S^m\times \re^n,g^m_0+g_E)$,  we make the following estimates for low dimensions. 

\begin{Theorem}
	\label{Iso}
The following bounds hold.

	\begin{equation}
		\label{s2xr2}
		I_{(S^2\times \re^2,g_0^2+g_E)}(v)\geq 0.88 I_{(S^4, 4.7 g_0^4)}(v)
	\end{equation}

	\begin{equation}
			\label{s2xr3}
	I_{(S^2\times \re^3,g_0^2+g_E)}(v)\geq 0.86I_{(S^5, 7.5 g_0^5)}(v)
\end{equation}
	 
		\begin{equation}
				\label{s3xr2}
		I_{(S^3\times \re^2,g_0^3+g_E)}(v)\geq 0.91  I_{(S^5, 4.9 g_0^5)}(v)
	\end{equation}
	
	\noindent Where $g_0^m$ denotes the round metric of the $m-$sphere and $g_E$ the flat metric of $\re^n$
\end{Theorem}

Estimating explicitly the isoperimetric profile of a Riemannian manifold has many applications  on differential geometry and geometric analysis.  For example, in some recent work of  Andreucci and  Tedeev \cite{Andre}, understanding the shape of the  isoperimetric profile on  some manifolds helps on the study of Sobolev and Hardy inequalities (see also \cite{Andre2}). As another example, Theorem 1.1  in \cite{PetRu2} relates lower bounds on the   isoperimetric profile of product manifolds, with lower  bounds on their Yamabe constant.%, see also \cite{Ilias}, \cite{Pet}.

For a closed   Riemannian manifold $(M^n,g)$, $n>2$, the solution of the Yamabe problem (cf. in \cite{Parker}) gives a metric $\tilde g$   for $M$ of constant scalar curvature and unit volume in the conformal class of $g$. These metrics are critical points of the Hilbert-Einstein functional restricted to conformal classes. The minima of the restricted functional are always realized as proved by the combined efforts of H. Yamabe \cite{Yamabe}, T. Aubin \cite{Aubin}, R. Schoen \cite{RSchoen}, and N. S. Trudinger \cite{Trudinger}, giving the solution of the Yamabe problem. These metrics are called Yamabe metrics and have constant scalar curvature.  The Yamabe constant is defined as the infimum of the total scalar curvature, restricted to the conformal class of $g$:
\begin{equation}
	(Y,[g])= \inf_{g\in [g]}\frac{\int_M s_g dvol_g}{Vol(M,g)^{\frac{n-2}{2}}}.
\end{equation}

\noindent Where $s_g$ denotes the scalar curvature of $g$, $[g]$ the conformal class of $g$ and $dvol_g$ the volume element of $g$. 
In general, to compute the Yamabe constant of a manifold is very difficult, as not all metrics of constant scalar curvature minimize the total scalar functional in a given conformal class, so it is not simple to asses when a metric of constant scalar curvature is a Yamabe metric. Therefore,   estimates of the Yamabe constant have been of interest. By  Theorem 1.1  in \cite{PetRu2},  we obtain the following as a corollary of Theorem \ref{Iso}.

\begin{Corollary}
	\label{Yamabe}
 
	The following holds.
	
		\begin{equation}
		Y(S^2\times \re^2,[g_0^2+g_E])\geq 0.78 Y(S^4)
	\end{equation}
	
		\begin{equation}
		Y(S^2\times \re^3,[g_0^2+g_E])\geq 0.75  Y(S^5)
	\end{equation}
	
		\begin{equation}
		Y(S^3\times \re^2,[g_0^3+g_E])\geq 0.83 Y(S^5)
	\end{equation}

\end{Corollary}

The Yamabe invariant  is defined as the supremum of all the Yamabe constants of a manifold and it was introduced by Kobayashi \cite{Kobayashi} and Schoen \cite{RSchoen2},

$$Y(M)=\sup_{\{[g]\}} Y(M,[g]).$$
It is an invariant of the smooth structure of $M$. Computing this invariant has also proven to be difficult and many efforts  have been made, see for example \cite{Aku}, \cite{ADH1}, \cite{ADH2} and \cite{Gursky}. See also \cite{LeBrun} for a detailed discussion of the Yamabe invariant on dimension 4.

Due to Theorem 1.1 in \cite{AFP}, the Yamabe constants of the  product of a manifold with positive scalar curvature and Euclidean space are general bounds to the Yamabe invariant of some product manifolds. We thus also get the following  as a Corollary of Theorem \ref{Iso}.

\begin{Corollary}
	\label{Yamabe2}
	Let $M^2$ be a 2-dimensional manifold and $N^3$ a 3-dimensional manifold.
	Then
	
		\begin{equation}
		Y(S^2\times M^2)\geq 0.785 Y(S^4)
	\end{equation}

	\begin{equation}
	Y(S^2\times N^3)\geq 0.75 Y(S^5)
	\end{equation}
	
		\begin{equation}
		Y(S^3\times M^2)\geq 0.83 Y(S^5)
	\end{equation}

\end{Corollary}

This result improves previously known bounds for the Yamabe invariant of some product manifolds of dimension 4. For example,   smaller bounds were shown  in \cite{PetRu1} and \cite{Pet2}, namely  that 
$	Y(S^2\times M^2)\geq 0.68 Y(S^4)$
and
 $	Y(S^2\times M^2)\geq 0.0006 Y(S^4),$
  respectively. On dimension 5, it was previously known (Theorem 1.5 in \cite{PetRu2}) that $	Y(S^2\times N^3)\geq 0.63Y(S^5)$ and $		Y(S^3\times M^2)\geq 0.75 Y(S^5)$. 
  
  The lower bounds in Corollary \ref{Yamabe} are also involved in a formula for surgery by B. Amman, M. Dahl and E. Humbert in \cite{ADH1}:
  
  \begin{Theorem}\textnormal{(Corollary 1.4 in \cite{ADH1})}
  	Let $M^n$ be a compact manifold. If $N$ is obtained from $M$ by a  $k$-dimensional
 surgery, $k \leq n - 3$, then
  $$Y(N)\geq \min\{Y(M), \Lambda_{n,k}\} $$
  where $ \Lambda_{n,k}>0$ is a dimensional constant that depends only on $n$ and $k$.
\end{Theorem}

Following the   estimations in \cite{ADH2}, Theorem \ref{Iso} and Corollary \ref{Yamabe} improve on previously known lower bounds for some of the $ \Lambda_{n,k}$, namely, we now have,
$\Lambda_{4,1}\geq 0.71 Y(S^4)\geq 43.9 $, $\Lambda_{5,1}\geq 0.718 Y(S^5) \geq 56.7$, and $\Lambda_{5,2}\geq0.62Y(S^5) \geq49$. As another application, recall that the range of the Yamabe invariant for simply connected manifolds of dimension 5 is limited, by a result of B. Ammann, M. Dahl and E. Humbert in \cite{ADH2}: 

 \begin{Theorem} \textnormal{(Corollary 5.3 in \cite{ADH2})}
Let $M$ be a compact simply connected manifold of dimension 5, then: $0.57  \ Y(S^5)\leq Y(M)\leq Y(S^5) $.
\end{Theorem}

Following the   estimations in \cite{ADH2}, the bounds in Corollary \ref{Yamabe} can be used to improve the last inequality to: $0.62 \   Y(S^5)\leq Y(M)\leq Y(S^5)$.

\section{Notation and background}

Existence and regularity of isoperimetric regions is a fundamental result due to the works of Almgren \cite{Almgren}, Gr\"uter \cite{Gruter}, Gonzalez, Massari, Tamanini \cite{Gonzalez}, (see also Morgan \cite{MorganBook}, Ros \cite{Ros}).

\begin{Theorem}
	\label{existence}
	Let $M^n$ be a compact Riemannian manifold, or non-compact with $M/G$ compact, being $G$ the
	isometry group of $M$. Then, for any $t$, $0 < t < V(M)$, there exists a
	compact region $\Omega \subset M$, whose boundary $\Sigma = \partial \Omega$ minimizes area among regions of
	volume $t$. Moreover, except for a closed singular set of Hausdorff dimension at most
	$n - 8$, the boundary $\Sigma$ of any minimizing region is a smooth embedded hypersurface
	with constant mean curvature.
	
\end{Theorem}

We will denote by $V_m=Vol(S^m,g_0^m)$ the m-volume of the $m$-sphere with the round metric and radius one. We denote  by 
$\gamma_m$ the  $m$-dimensional Euclidean isoperimetric constant. Recall that

$$\gamma_m=\frac{V_{m-1}}{Vol(D^m_1)^{\frac{m-1}{m}}}$$
\noindent where $D^m_1$ is a disk of radius 1 on the m-Euclidean space with the flat metric.
 %In particular $\gamma_4=2^{\frac{7}{4}}\sqrt{\pi}$ and $\gamma_5=(\frac{8}{3} \pi^2 5^4)^{\frac{1}{5}}$.

Note that the isoperimetric profile of $\re^n$ is given by $I_{(\re^n, g_E)}(v)=\gamma_n v^{\frac{n-1}{n}}$. Through the article we will denote by $f_n(v)$ the isoperimetric profile of $\re^n$.

For any closed Riemannian manifold $(M^m,g)$ we have 
$\lim_{v\rightarrow 0} \frac{I_{(M^m,g)}(v)}{v^{\frac{m-1}{m}}}=\gamma_m$.

The isoperimetric profile of the spherical cylinder $(S^m\times \re,  g_0^m+dt^2)$, $m\geq2$, was studied by Pedrosa in \cite{Pedro}. Isoperimetric regions are either ball type regions or cylindrical sections $S^m \times (a,b)$ of volume $V_m(b-a)$ and area $2V_m$.

The explicit formulas for the ball type regions $\Omega_h$ (Proposition 4.1 in \cite{Pedro}), are

\begin{equation}
	\label{pedrosa1}
	Vol^m(\partial \Omega_h)=2V_{m-1}\int_0^\eta \frac{\sin^{m-1}(y)}{\sqrt{1-(u_{m-1}(\eta,y))^2}}dy,
\end{equation}

\begin{equation}
	\label{pedrosa2}
	Vol^{m+1}(  \Omega_h)=2V_{m-1}\int_0^\eta \frac{\left(\int_0^y\sin^{m-1}(s) ds\right) u_{m-1}(\eta,y) }{\sqrt{1-(u_{m-1}(\eta,y))^2}}dy,
\end{equation}

\noindent for $\eta \in (0,\pi)$. Where

$$u_{n-1}=\frac{(\sin(\eta))^{n-1}/ \int_0^{\eta}(\sin(s))^{n-1}ds}{(\sin(y))^{n-1}/ \int_0^{y}(\sin(s))^{n-1}ds}.$$

and 
$$h=h_{n-1}(\eta)=\frac{(\sin(\eta))^{n-1}}{ \int_0^{\eta}(\sin(s))^{n-1}ds}.$$

For the m-sphere with the round metric $(S^m,g_0^m)$, isoperimetric regions are metric balls. The explicit formulas for these are:

\begin{equation}
	\label{sphere2}
	Vol^{m-1}(\partial \Omega_r)=V_{m-1}  \sin^{m-1}(r),
\end{equation}

\begin{equation}
	\label{sphere1}
	Vol^m(  \Omega_r)=V_{m-1} \int_0^r \sin^{m-1}(u)du,
\end{equation}

\noindent for $r \in [0,\pi]$.
The isoperimetric profile of the $m$-sphere  is therefore  known precisely. In particular, it is symmetric, concave and reaches its maximum at $v=\frac{vol(S^m,g_0)}{2}$.
%p. 18 of Bayle

There is homogeneity for the isoperimetric profiles. For a given manifold $(M^n,g)$, $n\geq2$, for any $\mu>0$, 

$$I_{(M,\mu g)}(v)=\mu^{\frac{n-1}{2}} I_{(M, g)}(\mu^{\frac{-n}{2}}v).$$

%Proposition 1.2.1 of Bayle

For regions of the type	$M^m\times D^n_R\subset M^m\times \re^n$, of volume $v>0$, where   $D^n_R$ is a disk in $\re^n$ of radius $R>0$,  it can be  checked by direct computation  that the area of its boundary is given by the function $F_{M,n}(v)$ in equation   (\ref{Fmn}).

\section{Isoperimetric profiles of $M^m\times \re^n$}

\begin{proof}
	
	 ({\it of Theorem \ref{Thm1}})
	
	We first follow a construction by F. Morgan in \cite{Morgan}, which estimates lower bounds for isoperimetric profiles of products. %We consider the product of $(M,g)$  and $(  \re^n, g_E)$. 
	Consider the product manifold  $(0,Vol(M))\times (0,\infty)\subset \re^2$ %. We give this 2-dimensional manifold 
	with a model metric in the sense of the Ros product Theorem (Theorem 3.7 in \cite{Ros}).   $(0,Vol(M))$  and $(0,\infty)$ will have  Euclidean Lebesgue Measure and Riemannian metric $\frac{1}{h(x)}dx$ and $\frac{1}{f_n(y)}dy$ respectively, where $h(x)$ is the concave lower bound of the isoperimetric profile of $(M,g)$ and $f_n(y)$ the isoperimetric profile of $\re^n$. 
	
	This should be  a model metric in order for the Ros product Theorem to work. It suffices to prove that in each interval, $(0, Vol(M))$ and $(0, \infty)$, intervals of the type $(0,t)$, $t>0$, minimize perimeter, among closed sets $S$ of given Euclidean length $t$. For the interval $(0, \infty)$ this is true since $f_n(x)$ is nondecreasing. For the interval $(0,Vol(M))$, we argue as follows. Let $S\subset(0,Vol(M))$ be a closed set of perimeter $t$,  and suppose $S$ is not of the type $(0,t)$; then it must be a locally finite collection of closed intervals and an interior interval must be at least borderline unstable,  because   $h(v)$ is   concave. Hence $S$ does not minimize perimeter.  %$\frac{1}{2}ds$.

	Minkowski content on the intervals $(0,Vol(M))$  and $(0,\infty)$ counts boundary points of intervals with density $h(y)$ and $f_n(x)$, respectively. Minkowski content on the stripe
	$(0,Vol(M))\times (0,\infty)$ has perimeter measured by 
	
	\begin{equation}
		\label{Morg}
		ds^2=h^2(y) dx^2+ f_n(x)^2 dy^2.
	\end{equation}

	It follows from the proof of the Ros Product Theorem that, for any $v>0$,  $I_{(M\times  \re^n)}(v)$ is bounded from below by the perimeter $P(E)$ of the boundary $\delta E$ of some region $E \subset (0,Vol(M))\times (0,\infty)$. The area of $E$, $A(E)$, satisfies $v= A(E)$, and $\delta E$ is a connected boundary curve along which $y$ is nonincreasing and $x$ is nondecreasing. The enclosed region $E$ is on the lower left of $\delta E$.  Hence
	
	$$P(E)\leq I_{(M \times \re^n)}(v)$$
	
	where
	
	\begin{equation}
		\label{PE}
		P(E) = \int_{\delta E} \sqrt{ h^2(y) dx^2 + f_n(x)^2 dy^2}.
	\end{equation}
	and the area of the region $E$ is given by
	\begin{equation}
	\label{Area}
	A(E) = \int \int_E dx \, dy.
		\end{equation}
	
	Since each term in the square root of eq. (\ref{PE})  is non-negative, we have 
	
	\begin{equation}
		\label{PE1}
		P(E) \geq  \int_{\delta E} f_n(x)  dy
	\end{equation}
	
	and
	\begin{equation}
		\label{PE2}
		P(E) \geq  \int_{\delta E}   h(y)   dx. 
	\end{equation}

%Let $\Omega$ be an isoperimetric, symmetrized, region on $S^m \times \re^n$, of volume $v$ and area of its boundary $\partial(\Omega)$, $A$. 
 Let $v>0$, with corresponding region $E \subset (0,Vol(M))\times (0,\infty)$ and boundary $\delta E$. For $t$, $0<t\leq 1$, consider the stripe

$$W_{t}= \left\{(x,y) | x \geq t\frac{v}{ Vol(M)} \right\}\subset  (0,\infty) \times (0, Vol(M)). $$
Naturally, the area of $W_{t}$ is infinite but for its complement we have $A(W_{t}^c)=t v$.

%Let $R>0$, such that the box $B_R=S^m\times D_R^n$, has volume $v$. Let $\Omega_2=\Omega\cap (B_R)^c$. Let $A_2=Vol^{m+n-1}(\partial\Omega\cap B_R^c))$

Fix $\alpha$, with $0<\alpha< 1$  as in the hypothesis. Suppose that for some $(x,y)$ in the boundary curve $\delta E$, such that $x = \alpha \frac{v}{Vol(M)}$, we have
  $y\geq \alpha Vol(M)$.

Note that in this case, %the area of $\Omega$, $a$, is bounded from below, 

$$I_{M\times \re^n}(v)\geq P(E) \geq   \int_{ \delta E} \sqrt{f_n^2(x) dy^2+h_m^2(y)dx^2} \geq  \int_{ \delta E \cap W_{\alpha}} \sqrt{f_n^2(x) dy^2+h_m^2(y)dx^2}  $$
$$ \geq f_n\left(\alpha \frac{v}{ Vol(M)}\right)\alpha  Vol(M)=(\gamma_n) \left(\alpha \frac{v}{ Vol(M)}\right)^{\frac{n-1}{n}} \alpha Vol(M)$$
$$=  (\gamma_n Vol(M)^{\frac{1}{n}})\alpha^{2-\frac{1}{n}} v^{\frac{n-1}{n}} = \alpha^{2-\frac{1}{n}} C_{M,n} v^{\frac{n-1}{n}}=\alpha^{2-\frac{1}{n}} F_{M,n}(v)$$

	\noindent as stated in the Theorem.
	
Suppose that otherwise,  for all $(x,y)$ in the boundary curve $\delta E$, such that $x= \alpha \frac{v}{Vol(M)}$, we have
 $y< \alpha Vol(M)$. We   show next that this cannot be the case for $v$ big enough.
	
 %Note that $A(W_1^c \cap E^c)=v-\tilde v$
	
	Since $y$ is non increasing, then, for all $(x,y) \in E$, such that  $x\geq \alpha \frac{v}{Vol(M)}$, we have $y< \alpha Vol(M)$. 
	This implies that for  the rectangle
	
	$$R=\{(x,y)|\alpha \frac{v}{Vol(M)}\leq x<  \frac{v}{Vol(M)} \  \ \textnormal{and} \   \  \alpha Vol(M) \leq y <   Vol(M)\},$$
\noindent we have  $R\subset E^c$ and  $R\subset W_1^c$. Also, the Area of $R$ can be computed to be,	$A(R)= \left( (1-\alpha) \frac{v}{Vol(M)}\right) \left( (1-\alpha) Vol(M) \right)=(1-\alpha)^2 v$.

Now, since  
% $$E=(W_1\cap E) \cup (W_1^c \cap E)$$

%$$W_1^c=(W_1^c\cap E) \cup (W_1^c \cap E^c)$$

 $$v=A(E)=A(W_1\cap E) + A(W_1^c \cap E),$$
 and
 
  $$v=A(W_1^c)=A(W_1^c\cap E) +A(W_1^c \cap E^c),$$
\noindent then, by letting  $\tilde v=A(E\cap W_1)$, we have,  $A(W_1^c \cap E^c)=v-(v-\tilde v)=\tilde v.$
  
  Together with $R\subset (W_1^c \cap E^c )$, this implies $A(R)= (1-\alpha)^2 v\leq \tilde v $, that is
\begin{equation}
	\label{propleft}
	\frac{\tilde v}{v}\geq (1-\alpha)^2.
\end{equation}

	On the other hand, by hypothesis, $h_m$ is a concave function with $h_m(0)=h_m(Vol(M))=0$. Hence, there is some    $k>0$, such that 
	$h_m(\alpha Vol(M))=k \alpha Vol(M)$, and $h_m(y)\geq k y$ for $y\leq  \alpha Vol(M)$. This yields %Of course, $k= \frac{\alpha Vol(M)}{h_m(\alpha Vol(M))}$.
	
%If 	$y< \alpha V_m$  for some $x= \alpha \frac{V}{V_m}$.

$$I_{M\times \re^n}(v)\geq P(E) \geq  \int_{ \delta E} \sqrt{f_n^2(x) dy^2+h_m^2(y)dx^2}\geq     \int_{ \delta E \cap W_1} \sqrt{f_n^2(x) dy^2+h_m^2(y)dx^2}$$
$$\geq \int_{\delta E \cap W_1}  h_m(y)dx\geq k \int_{\delta E \cap W_1}  y dx \geq k \tilde v$$

\noindent because $\tilde v=A(E\cap W_1)=\int_{\delta E \cap W_1}  y dx$. Thus, since we always have $F_{M,n}(v)\geq I_{(M \times \re^n)}(v)$, we get,% if we consider regions of the type $M^m\times D_R\subset M^m\times \re^n$ of volume $v$, we have,
$$C_{M,n} v^{\frac{n-1}{n}}=F_{M,n}(v)\geq I_{(M \times \re^n)}(v)\geq k \tilde v$$
That is

\begin{equation}
	\label{propright}
	\frac{C_{M,n}}{k} \frac{1}{v^{\frac{1}{n}}}\geq \frac{\tilde v}{v}.
\end{equation}

Finally, from (\ref{propleft}) and  (\ref{propright}):$$ 	\frac{C_{M,n}}{k} \frac{1}{v^{\frac{1}{n}}}\geq\frac{\tilde v}{v}\geq (1-\alpha)^2,$$

which cannot occur for $v\geq v_0=\left(\frac{C_{M,n}}{k(1-\alpha)^2}\right)^n$.

\end{proof}

The following lemma will let us extend a known lower bound for  $I_{S^m\times \re^n}(v)$.

\begin{Lemma}
	\label{x0y0}
%Let $(S^m,g)$ a compact manifold, $m>1$.	
Suppose that for some  $x_0,y_0>0$, we have $y_0\leq I_{S^m\times \re^n}(x_0)$. Then, for all $v>x_0$:
	
	$$ \left(\frac{y_0}{x_0^{\frac{n-1}{n}}}\right) v^{\frac{n-1}{n}} \leq  I_{S^m\times \re^n}(v). $$
\end{Lemma}

\begin{proof}
	Note that both $(S^m,g_0)$ and $(\re^n,g_E)$ are model metrics in the sense of the Ros product Theorem \cite{Ros}. This implies that   isoperimetric regions can be symmetrized with respect to both factors of the product $S^m\times \re^n$. This is, each slice $\Omega \cap ( \{p\} \times \re^n)$, $p\in S^m$ can be replaced by an $n-$disk in $\re^n$ of some radius $R_p\geq0$, $D_{R_p}$, such that $Vol^{n}(\Omega \cap ( \{p\} \times \re^n))=Vol^{n}(D_{R_p})$, and each slice $\Omega \cap ( S^m \times \{q\} )$, $q\in \re^n$ can be replaced by a (geodesic) $m-$ball in $S^m$ of some radius $t_q\geq0$, $B_{t_q}$, such that $Vol^{m}(\Omega \cap ( S^m \times \{q\} ))=Vol^{m}(B_{t_q})$. 
	
	As  in the proof of Theorem \ref{Thm1}, consider the product manifold $(0,Vol(M)) \times (0,\infty) \subset \re^2$
	where 	 $(0,Vol(M))$  and $(0,\infty)$  have  Euclidean Lebesgue Measure and Riemannian metric $\frac{1}{h(y)}dy$ and $\frac{1}{f_n(x)}dx$ respectively, where	$f_n(x)=\gamma_n x^{\frac{n-1}{n}}$ is the isoperimetric profile of $\re^n$ and $h_m(x)$ the isoperimetric profile of $S^m$.
	
Let $\Omega$ be an isoperimetric, symmetrized (with respect to both  $(S^m,g_0)$ and $(\re^n,g_E)$), region of volume $v$. As argued in the proof of Theorem \ref{Thm1},  there is a corresponding region $E\subset (0,Vol(S^m))\times (0, \infty)$ such that 
	
	$$v=Vol^{m+n}(\Omega)=A(E)=\int \int_E dy dx.$$
	\noindent And, moreover, since both $(S^m,g_0)$ and  $(\re^n,g_E)$ are model metrics, and $\Omega$ is symmetrized with respect to both $S^m$ and $\re^n$, we actually have equality for the perimeter of the boundary:
	
	$$Vol^{m+n-1}(\partial \Omega)= P(\delta E) =\int_{\delta E} \sqrt{f_n^2(x) dy^2+h_m^2(y) dx^2}$$
	\noindent

Now,  for fixed $r$,  $0<r\leq1$, we can construct a new region $\Omega_r$, from $\Omega$, in the following way.	Since each slice $\Omega \cap ( \{p\} \times \re^n)$, $p\in S^m$ is an $n-$disk of some radius $R_p$, we may replace each of these $n-$disks with a new $n-$disk of radius $r^{\frac{1}{n}}R_p$, for each $p \in S^m$. Of course, if the slice $\Omega \cap ( \{p\} \times \re^n)$ is empty we leave it that way.

	Let $$E_r=\{(rx,y)\in (0,Vol(S^m))\times (0, \infty)|(x,y)\in E\}.$$
	
Consequently,	the volume of $\Omega_r$ is 
	$$Vol^{m+n}(\Omega_r)=A(E_r)=\int \int_E dy (r dx) = r \int \int_E dy dx =r Vol^{m+n}(\Omega),$$

		\noindent and the $m+n-1$ volume  of its boundary $\partial \Omega_r$, is 
	$$Vol^{m+n-1}(\partial \Omega_r)= P(\delta E_r) =\int_{\delta E_r} \sqrt{f_n^2(rx) dy^2+h_m^2(y) r^2dx^2}$$

	We then have,
$$	Vol^{m+n-1}(\partial \Omega_r)= \int \sqrt{r^{\frac{2(n-1)}{n}}f_n^2(x) dy^2+r^2 h_m^2(y)dx^2}$$
$$= r^{\frac{n-1}{n}}\int \sqrt{f_n^2(x) dy^2+r^{\frac{2}{n}} h_m^2(y)dx^2}\leq  r^{\frac{n-1}{n}} \int \sqrt{f_n^2(x) dy^2+ h_m^2(y)dx^2},$$ 
\noindent since $r^{\frac{2}{n}}\leq1$.	This implies $	Vol^{m+n-1}(\partial\Omega_r) \leq 	r^{\frac{n-1}{n}} 	Vol^{m+n-1}(\partial\Omega).$

	Now, suppose as in the hypothesis that $y_0\leq I_{M^m\times \re^n}(x_0)$. Let $v>x_0$. Let $\Omega$ be an isoperimetric, symmetrized, region of $M^m\times \re^n$ of volume $v$.
	
	Let $r=\frac{x_0}{v}<1$. We construct $\Omega_r$ as before,  so that
	
	$$Vol^{m+n}(\Omega_r)=r v=x_0$$
	and 
	$$Vol^{m+n-1}(\partial \Omega_r)\leq r^{\frac{n-1}{n}}Vol^{m+n-1}(\partial \Omega)=r^{\frac{n-1}{n}}I_{M^m\times \re^n}(v)$$
	since $\Omega$ is isoperimetric of volume $v$. This yields
	$$y_0\leq I_{M^m\times \re^n}(x_0)\leq Vol^{m+n-1}(\partial \Omega_r)\leq r^{\frac{n-1}{n}}I_{M^m\times \re^n}(v).$$
	
 By recalling that $r=\frac{x_0}{v}$, we conclude the  lemma.
	
\end{proof}
\begin{comment}

\begin{Corollary}
Let $(M^m,g)$ a compact manifold, $m>1$.	Suppose $M^m\times D_{R_0}$ is an isoperimetric region of $M^m\times \re^n$,   then $M^m\times D_R$ is also isoperimetric, for $R>R_0$.
\end{Corollary}
\begin{proof}
	Let $x_0=Vol^{m+n}(M^m\times D_{R_0})=Vol(M) \bar V_n R_0^{n}$ (where $\bar{V}_n=Vol^{n}(B^n(0,1))$) and $y_0=Vol^{m+n-1}(\partial (M^m\times D_{R_0}))=Vol(M) V_{n} R_0^{n-1}$. By hypothesis
	$$I_{M^m\times \re^n}(x_0)=y_0=Vol(M) V_{n} R_0^{n-1}.$$
	Using the last Lemma, we get, for $R>R_0$:
		$$I_{M^m\times \re^n}(Vol^{m+n}(M^m\times D_{R}))=I_{M^m\times \re^n}(Vol(M) \bar V_n R^n)\geq$$
		$$\left(\frac{Vol(M) V_{n} R_0^{n-1}}{(Vol(M) \bar V_n R_0^n)^{\frac{n-1}{n}}}\right) (Vol(M) \bar V_n R^n)^{\frac{n-1}{n}}=Vol(M) V_{n}R^{n-1}=Vol^{m+n-1}(\partial (M^m\times D_{R}))$$
And being $M^m\times D_{R}$ an actual region, we get the other inequality.
%	$$I_{S^m\times \re^n}(V_m \bar V_n R^n)=I_{S^m\times \re^n}(Vol(S^m\times D_{R}))\leq Area(\partial (S^m\times D_{R}))=V_m V_{n}R^{n-1}$$
\end{proof}
\end{comment}

We will use the following observation.

 \begin{Lemma}
	\label{Fbeta}
If a non-negative function $J(v)=I^{\frac{n}{n-1}}(v)$,  $n>1$, with $J(0)=0$, is concave ($ J''(v)\leq 0$), then the function $\frac{I(v)}{v^{\frac{n-1}{n}}}$ is non increasing.

\end{Lemma}

\begin{proof}
	Note that $K(v)=J'(v)v-J(v)$ is decreasing for $v\geq0$, since $K'(v)=J''(v)v\leq0$, for $v\geq0$.  Moreover, $K(0)=0$, by hypothesis on $J(v)$.  	This implies that $K(v)\leq0$ and thus $\frac{d}{dv}\left(\frac{J(v)}{v}\right)=\frac{J'(v)v-J(v)}{v^2}=\frac{K(v)}{v^2}\leq 0$. Of course, being $\frac{J(v)}{v}$ non increasing implies that $\left(\frac{J(v)}{v}\right)^{\frac{n-1}{n}}=\frac{I(v)}{v^{\frac{n-1}{n}}}$ is non increasing.

\end{proof}

\begin{Lemma}
	\label{smalls}
	Let $(M^m,g)$, $(N^m,h)$ be two complete Riemannian manifolds, $m>1$, of non-negative Ricci curvature.  Let $v_0, k>0$, such that 
	
	$$ I_{(M,g)}(v_0)>k v_0^{\frac{m-1}{m}}.$$
	
	Then,  for $0 \leq v\leq v_0$:

	$$  \frac{k}{\gamma_m} I_{(N,h)}(v) < I_{(M,g)}(v),$$

\end{Lemma}

\begin{proof}
	Being the manifolds of the hypothesis of non-negative Ricci curvature, by a result of V. Bayle (see  \cite{Bayle}, Theorem 2.2.1 for the compact and Corollary 2.2.8 for the non-compact version), the renormalized isoperimetric profiles $J_{(M,g)}(v)=I^{\frac{m}{m-1}}_{(M,g)}(v)$ and $J_{(N,h)}(v)=I^{\frac{m}{m-1}}_{(N,h)}(v)$	are concave. 	Thus, by the preceding Lemma, the functions $\frac{I_{(M,g)}(v)}{v^{\frac{m-1}{m}}}$ and $ \frac{I_{(N,h)}(v)}{v^{\frac{m-1}{m}}}$ are non increasing.	This implies,  for $v\geq 0$:
	\begin{equation}
		\label{gamma1}
		\gamma_m =	\lim_{v\rightarrow 0} \frac{I_{(N,h)}(v)}{v^{\frac{m-1}{m}}} \geq \frac{I_{(N,h)}(v)}{v^{\frac{m-1}{m}}}
	\end{equation} 
	Also   (using the hypothesis) we have, for $v\leq v_0$:  
	\begin{equation}
		\label{gamma2}
		\frac{I_{(M,g)}(v)}{v^{\frac{m-1}{m}}}	\geq \frac{I_{(M,g)}(v_0)}{v_0^{\frac{m-1}{m}}} >k
	\end{equation}

	Then, by combining (\ref{gamma1}) and  (\ref{gamma2}),   for $0\leq v\leq v_0$:
	
	\begin{equation}
		\frac{1}{\gamma_m}I_{(N,h)}(v)\leq	 v^{\frac{m-1}{m}} <	\frac{1}{k}	 I_{(M,g)}(v), 
	\end{equation}
	
	\noindent which implies the Lemma.
\end{proof}

Hence,  a punctual bound for the isoperimetric profile of $S^m\times \re^n$  can be extended onwards and backwards.

\begin{Corollary}
	Suppose there are some $x_0,y_0\geq 0$, such that $I_{(S^m\times \re^n,g_0^m+g_E)}(x_0)>y_0$. Then, for $v\leq x_0$:
	
	$$\left(\frac{y_0}{\gamma_{m+n} \ x_0^{\frac{m+n-1}{m+n}}}\right) I_{(S^{m+n},g_0)}(v) \leq I_{S^m\times \re^n}(v) $$
	
	while for $v\geq x_0$:
	
		$$\left(\frac{y_0}{ x_0^{\frac{n-1}{n}}}\right)v^{\frac{n}{n-1}} \leq I_{S^m\times \re^n}(v) $$
		 
\end{Corollary}

\begin{proof}
	The first inequality  follows from lemma \ref{smalls}, by letting $v_0=x_0$, and $k=\frac{y_0}{x_0^{\frac{m+n-1}{m+n}}}$. The second one, directly from lemma \ref{x0y0}.
\end{proof}

\begin{figure}
		 	\includegraphics[scale=.2]{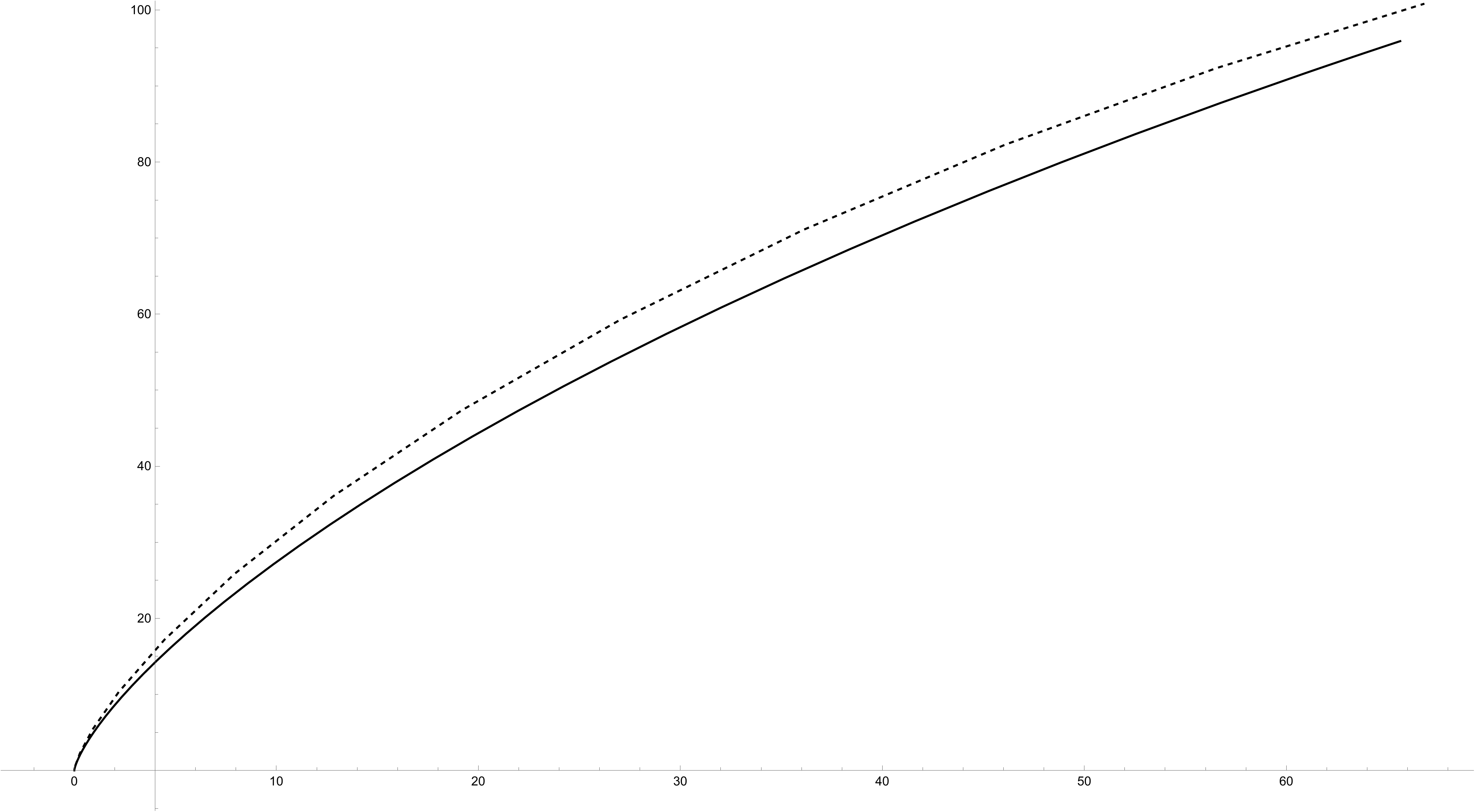}
	\caption{Comparison of   $ I_{(S^3\times \re, 2 (g_0^3+dt^2))}(v)$ (dashed) and  $0.886 I_{(S^5, 7.85 g_0^5)}(v)$ (continuous) for $4 \leq v\leq 65 $.}
	\label{S44to65}
\end{figure}

\begin{figure}
	\includegraphics[scale=.2]{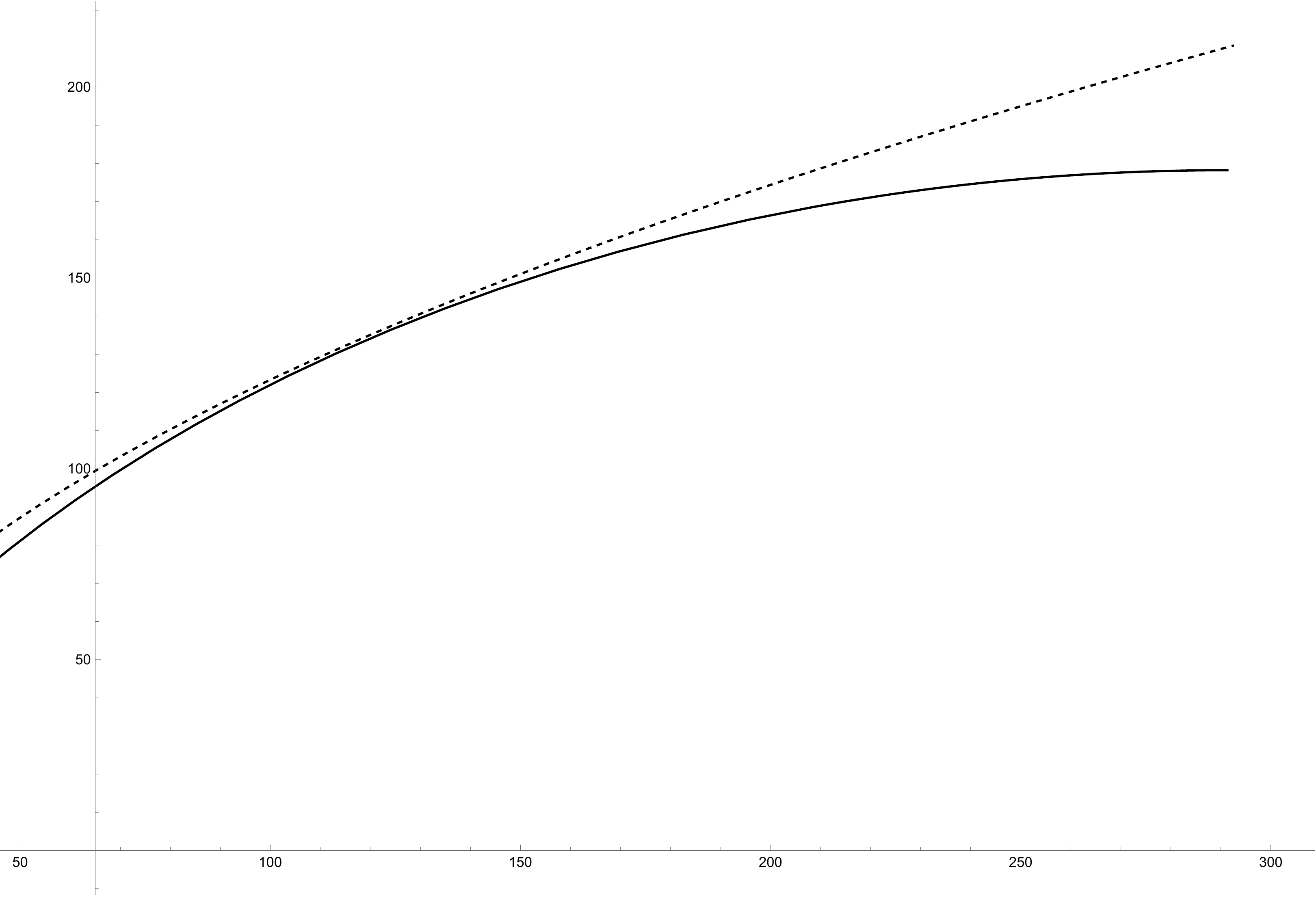}
	\caption{Comparison of   $12.32 v^{\frac{1}{2}}$ (dashed) and  $(0.886 )I_{(S^5, 7.85 g_0^5)}(v)$ (continuous) for $65\leq v\leq 291 $.}
		\label{S465to291}
\end{figure}

We will make use of the following result from  \cite{PetRu1}.

\begin{Lemma}\textnormal{(Lemma 2.1 of \cite{PetRu1})}
	\label{S2xR2}
	For all $v\geq0$
	
	$$ I_{(S^2\times \re^2,g_0^2+dt^2)}(v)\geq 	I_{(S^3\times \re, 2 (g_0^3+dt^2))}(v)$$
	
\end{Lemma}

We now prove the inequalities of Theorem \ref{Iso}. Recall that the isoperimetric profiles of $(S^m,g_0^m)$ and $(S^m\times\re,g_0^m+dt^2)$ are known (eqs. (\ref{sphere2}),(\ref{sphere1}) and (\ref{pedrosa1}),(\ref{pedrosa2}), respectively) so that comparisons with these profiles  can be done through direct   computations when needed. These will be shown on graphs.

We begin with the proof of inequality (\ref{s2xr2}) of Theorem \ref{Iso} in the next two lemmas. 

 \begin{Lemma}
	\label{22a}
	For $ v\leq 65$, $ I_{(S^2\times \re^2,g_0^2+dt^2)}(v)\geq 0.886 I_{(S^4, 4.7 g_0^4)}(v).$
	
\end{Lemma}
\begin{proof}
	We will prove 
	\begin{equation}
		\label{v65}
		I_{(S^3\times \re, 2 (g_0^3+dt^2))}(v)\geq 0.866  I_{(S^4, 4.7 g_0^4)}(v)
	\end{equation}
	for $v\leq 65$. The Lemma will then follow from inequality of Lemma \ref{S2xR2}.
	
	By direct computation,   $I_{(S^3\times \re, 2 (g_0^3+dt^2))}(4)>(5.5) 4^{\frac{3}{4}}$, which satisfies the hypothesis of Lemma \ref{smalls}.  Hence, for $v\leq 4$:
	
	$$I_{(S^3\times \re, 2 (g_0^3+dt^2))}(v)>\frac{\gamma_4}{5.5}  I_{(S^4, 4.7 g_0^4)}(v)> 0.866  I_{(S^4, 4.7 g_0^4)}(v).$$

	Finally,  for $4 \leq v\leq 65$, inequality (\ref{v65})    can be checked by  direct  computations, since these isoperimetric profiles are known. We provide the graphics (see figure \ref{S44to65}).
	
\end{proof}

 \begin{Lemma}
	\label{22b}
	For $ v\geq 65$, $ I_{(S^2\times \re^2,g_0^2+dt^2)}(v)\geq 0.886 I_{(S^4, 4.7 g_0^4)}(v).$
	
\end{Lemma}

\begin{proof}
	
	By direct computation, using equations (\ref{pedrosa1})  and (\ref{pedrosa2}) 
	$$ I_{(S^3\times \re, 2 (g_0^3+dt^2))}(65)\approx99.4$$
So that  $x_0=65$ and $y_0=99.4$   satisfy the hypothesis of Lemma \ref{x0y0}, since

$$99.4\leq  I_{(S^3\times \re, 2 (g_0^3+dt^2))}(65)\leq I_{(S^2\times \re^2,g_0^2+dt^2)}(65) $$
 Lemma \ref{x0y0}, then yields

\begin{equation}
	\label{vsqrt}
	 12.32 v^{\frac{1}{2}}\leq I_{(S^2\times \re^2,g_0^2+dt^2)}(v)
\end{equation}
 for $v\geq 65$. 
 The isoperimetric profile of $(S^4, 4.7 g_0^4)$ is given by eqs. (\ref{sphere2}) and (\ref{sphere1}).
  %$$I_{(S^4, 5 g_0^4)}(\frac{8}{3}(4.7 \pi)^2(2+\cos(r))\sin^4(\frac{r}{2}))= 4.7^{\frac{3}{2}}2 \pi^2\sin^3(r)$$
 Using direct   computations we check that,
 \begin{equation}
 	\label{numeric}
  0.886 I_{(S^4, 5 g_0^4)}(v)< 12.32 v^{\frac{1}{2}} 
\end{equation}

\noindent for $v$, $65 \leq v\leq 291$  (see figure \ref{S465to291}).
Note that at $\frac{25 V_4}{2}  \approx 290.69$,  $I_{(S^4, 5 g_0^4)}(v)$ attains its maximum, while $12.32 v^{\frac{1}{2}}$ is nondecreasing. This implies the inequality (\ref{numeric})  is still valid even for  $v\geq \frac{25 V_4}{2}  $. Inequalities (\ref{vsqrt}) and (\ref{numeric}) imply the Lemma. 
\end{proof}

%%%%%%%%%%%%%%%%%%%%%%%%%%%%%%%%%%%%%%%%%%%%%%%%%%%%%%%%%%%%%%%%%%%%%%%%%%%%%%%%%%%%%%%%%%%%%%
%%%%%%%%%%%%%%%%%%%%%%%%%%%%%%%%%%%%%%%%%%%%%%%%%%%%%%%%%%%%%%%%%%%%%%%%%%%%%%%%%%%%%%%%%%%%%%
%%%%%%%%%%%%%%%%%%%%%%%%%%%%%%%%%%%%%%%%%%%%%%%%%%%%%%%%%%%%%%%%%%%%%%%%%%%%%%%%%%%%%%%%%%%%%%

We will use the following results from \cite{PetRu2}.

\begin{Lemma}\textnormal{(Corollary 3.2 of \cite{PetRu2})}
	\label{Is4r1}
For all $v\geq 0$, 
	$$I_{(S^3\times \re^2,  (g_0^2+dt^2))}(v) \geq 0.99 I_{(S^{4}\times \re, 2^{\frac{3}{2}} (g_0^{4}+dt^2))}(v)$$
\end{Lemma}

\begin{Lemma}\textnormal{(Lemma 4.1 of \cite{PetRu2})}
	
	\label{Is4r2}
	For all $v\geq 0$, we have $I_{(S^2\times \re^3,  (g_0^2+dt^2))}(v) \geq 0.99 I_{(S^{4}\times \re, 2^{\frac{5}{3}} (g_0^{4}+dt^2))}(v).$
\end{Lemma}

\begin{figure}
	\includegraphics[scale=.25]{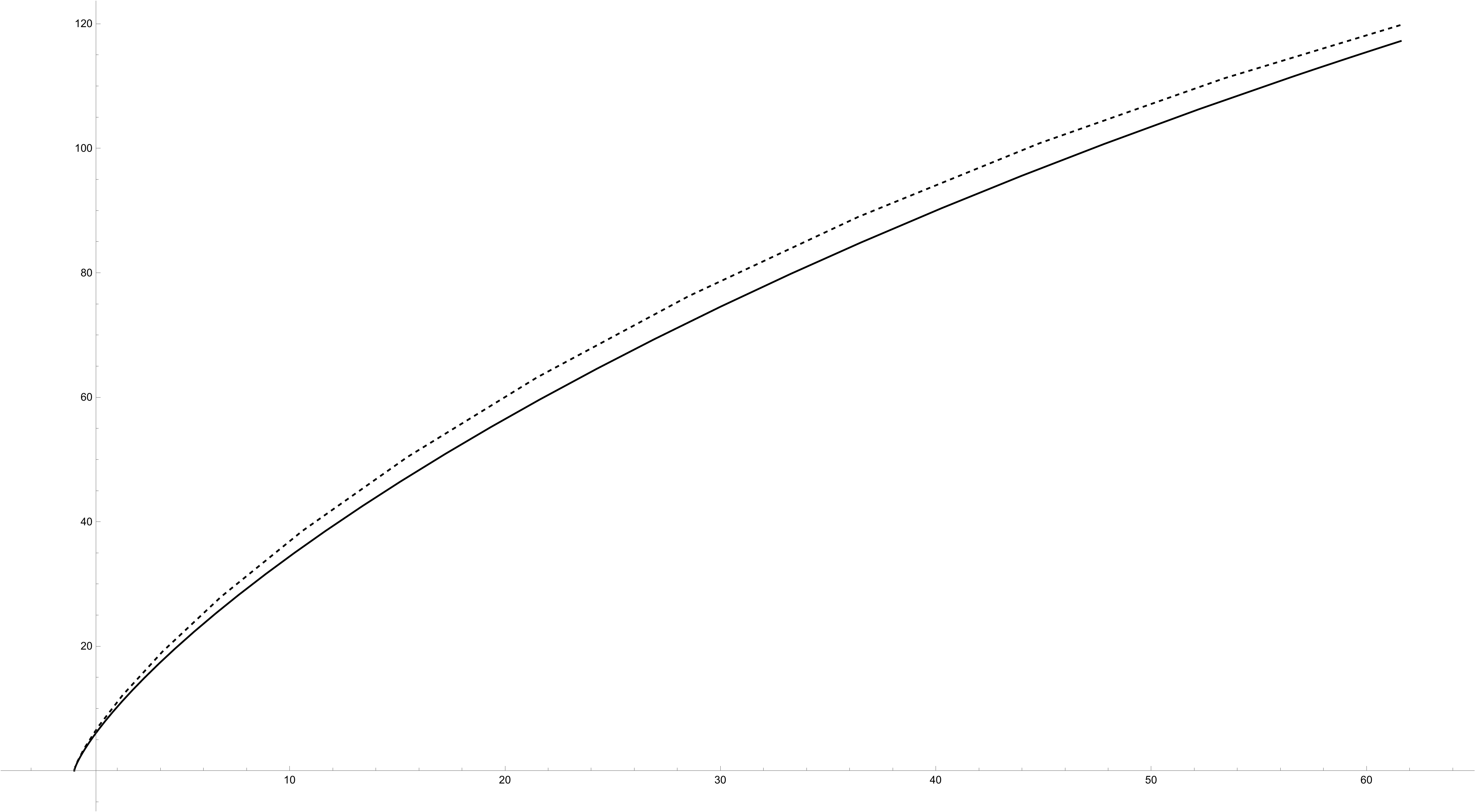}
		\caption{Comparison of   $ I_{(S^4\times \re, 2^{\frac{3}{2}} (g_0^4+dt^2))}(v)$ (dashed) and  $(0.91 ) I_{(S^5, 2.77 g_0^5)}(v)$ (continuous) for $1\leq v\leq 60 $.}
		\label{S51to60}
\end{figure}

\begin{figure}
	\includegraphics[scale=.4]{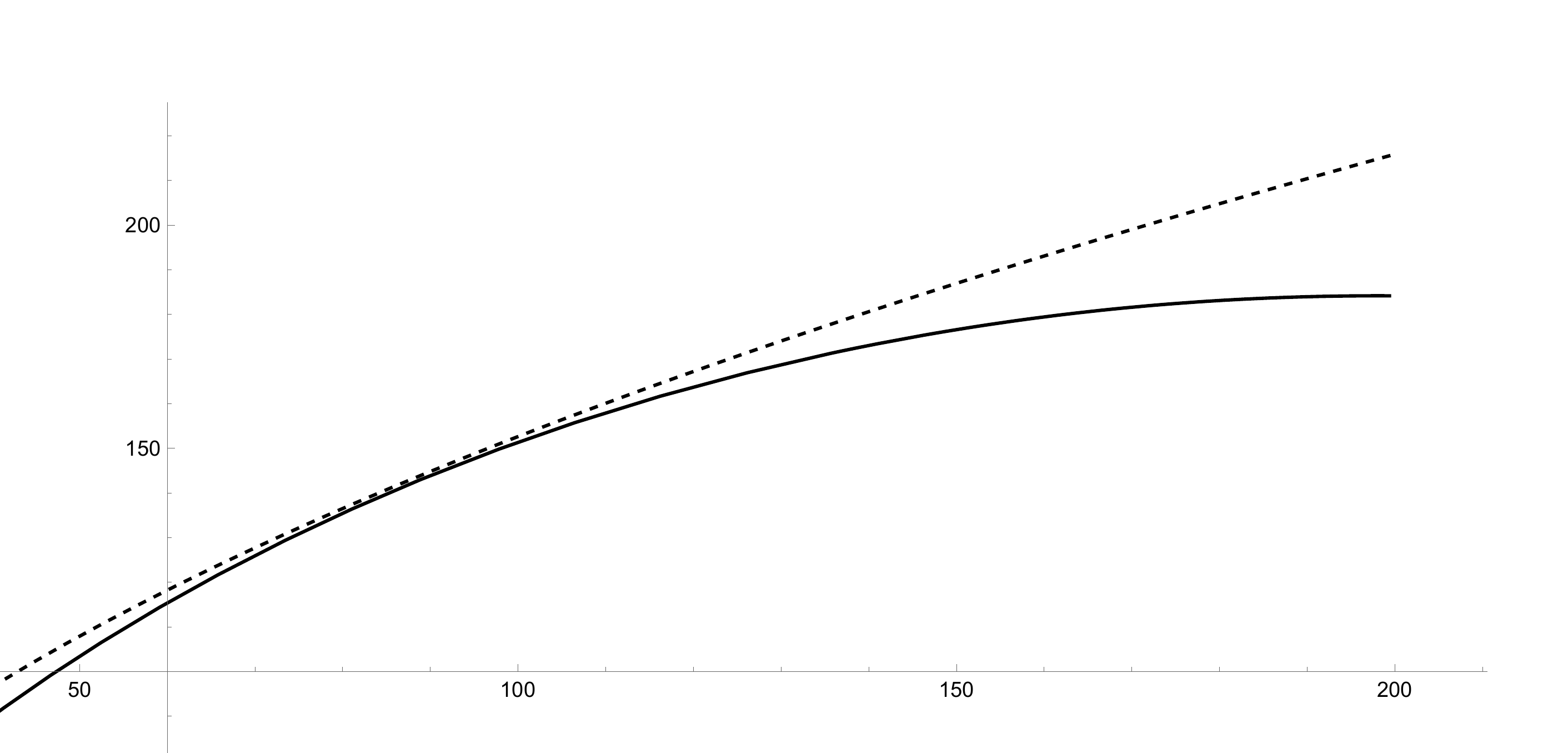}
		\caption{Comparison of   $ 15.26 v^{\frac{1}{2}}$ (dashed) and  $(0.91 )I_{(S^5, 2.77 g_0^5)}(v)$ (continuous) for $60\leq v\leq 200 $.}
		\label{S560to199}
\end{figure}

We now prove inequalities (\ref{s3xr2}) and (\ref{s2xr3}) from Theorem \ref{Iso}.

 \begin{Lemma}
	\label{32}
	For $v\geq0$, $ I_{(S^3\times \re^2,  (g_0^2+dt^2))}(v) \geq (0.91) I_{(S^5, 2.77 g_0^5)}(v).$
	
\end{Lemma}
\begin{proof}

%	$\mu_1=2^{\frac{3}{2}}$
%	$w_1=6.5$
%	$m=4$
	
	\textbf{Case 1:}  For $0\leq v\leq1$.

By direct computation, $0.99 I_{(S^{4}\times \re, 2^{\frac{3}{2}} (g_0^4+dt^2))}(1)>   6.5.$
Which we use as hypothesis of Lemma \ref{smalls}, to obtain: for $0\leq v\leq v_1$,
\begin{equation}
	\label{32case1}
0.99 I_{(S^{4}\times \re, 2^{\frac{3}{2}} (g_0^{4}+dt^2))}(v)>  \frac{ 6.5}{\gamma_5} I_{(S^{5},2.77 g_0^{5})}(v)>(0.91 )I_{(S^{5},2.77 g_0^{5})}(v).
\end{equation}

From this   and lemma \ref{Is4r1}, it follows that  the lemma is satisfied for $v\leq 1$.

\textbf{Case 2:}  For $1\leq v\leq60$.
	
Since the isoperimetric profiles $ I_{(S^{4}\times \re, 2^{\frac{3}{2}} (g_0^{4}+dt^2))}(v)$ and $I_{(S^{5},2.77 g_0^{5})}(v)$ are known, it follows from direct   computations that
$$	0.99 I_{(S^{4}\times \re, 2^{\frac{3}{2}} (g_0^{4}+dt^2))}(v)> (0.91 )I_{(S^{5},2.77 g_0^{5})}(v)$$
	
\noindent   for $1\leq v\leq60$, see figure \ref{S51to60}.
From this and lemma \ref{Is4r1}, the case is satisfied.

	\textbf{Case 3:} For $v \geq 60$.

We have 
$$I_{(S^3\times \re^2,  (g_0^2+dt^2))}(60)>	0.99 I_{(S^{4}\times \re, 2^{\frac{3}{2}} (g_0^{4}+dt^2))}(60)>118.245$$

\noindent where the first inequality follows from   lemma \ref{Is4r1}, and the second from direct computation. 
We may then use $x_0=60$, $y_0=118.245$ as the hypothesis for lemma \ref{x0y0}, and obtain,  for $v\geq 60$,
$I_{(S^3\times \re^2,  (g_0^2+dt^2))}(v)>15.26 v^{\frac{1}{2}}.$

On the other hand, by direct  computations (see figure \ref{S560to199}),  $15.26 v^{\frac{1}{2}}> (0.91 )I_{(S^{5},2.77 g_0^{5})}(v)$,
 for $60\leq v \leq 200$. So that combined, these last two inequalities give, for $60\leq v \leq 200$,

\begin{equation}
	\label{last32}
I_{(S^3\times \re^2,  (g_0^2+dt^2))}(v)>(0.91 )I_{(S^{5},2.77 g_0^{5})}(v)
\end{equation}

Finally, we note that $I_{(S^{5},2.77 g_0^{5})}(v)$ reaches its maximum at $v=\frac{Vol((S^{5},2.77 g_0^{5}))}{2}\approx198.4$, while $I_{(S^3\times \re^2,  (g_0^2+dt^2))}(v)$ is non-decreasing, by Lemma \ref{x0y0}. This implies that eq. (\ref{last32}) is still valid for $v \geq 200$.

\end{proof}
%%%%%%%%%%%%%%%%%%%%%%%%%%%%%%%%%%%%%%%%%%%%%%%%%%%%%%%%%%%%%%%%%%%%%%%%%%%%%%%%%%%%%%%%%%%%%%%%%%%%%%%%%%%%%%%
%%%%%%%%%%%%%%%%%%%%%%%%%%%%%%%%%%%%%%%%%%%%%%%%%%%%%%%%%%%%%%%%%%%%%%%%%%%%%%%%%%%%%%%%%%%%%%%%%%%%%%%%%%%%%%%
%%%%%%%%%%%%%%%%%%%%%%%%%%%%%%%%%%%%%%%%%%%%%%%%%%%%%%%%%%%%%%%%%%%%%%%%%%%%%%%%%%%%%%%%%%%%%%%%%%%%%%%%%%%%%%%
%%%%%%%%%%%%%%%%%%%%%%%%%%%%%%%%%%%%%%%%%%%%%%%%%%%%%%%%%%%%%%%%%%%%%%%%%%%%%%%%%%%%%%%%%%%%%%%%%%%%%%%%%%%%%%%

\begin{figure}
	
	\includegraphics[scale=.2]{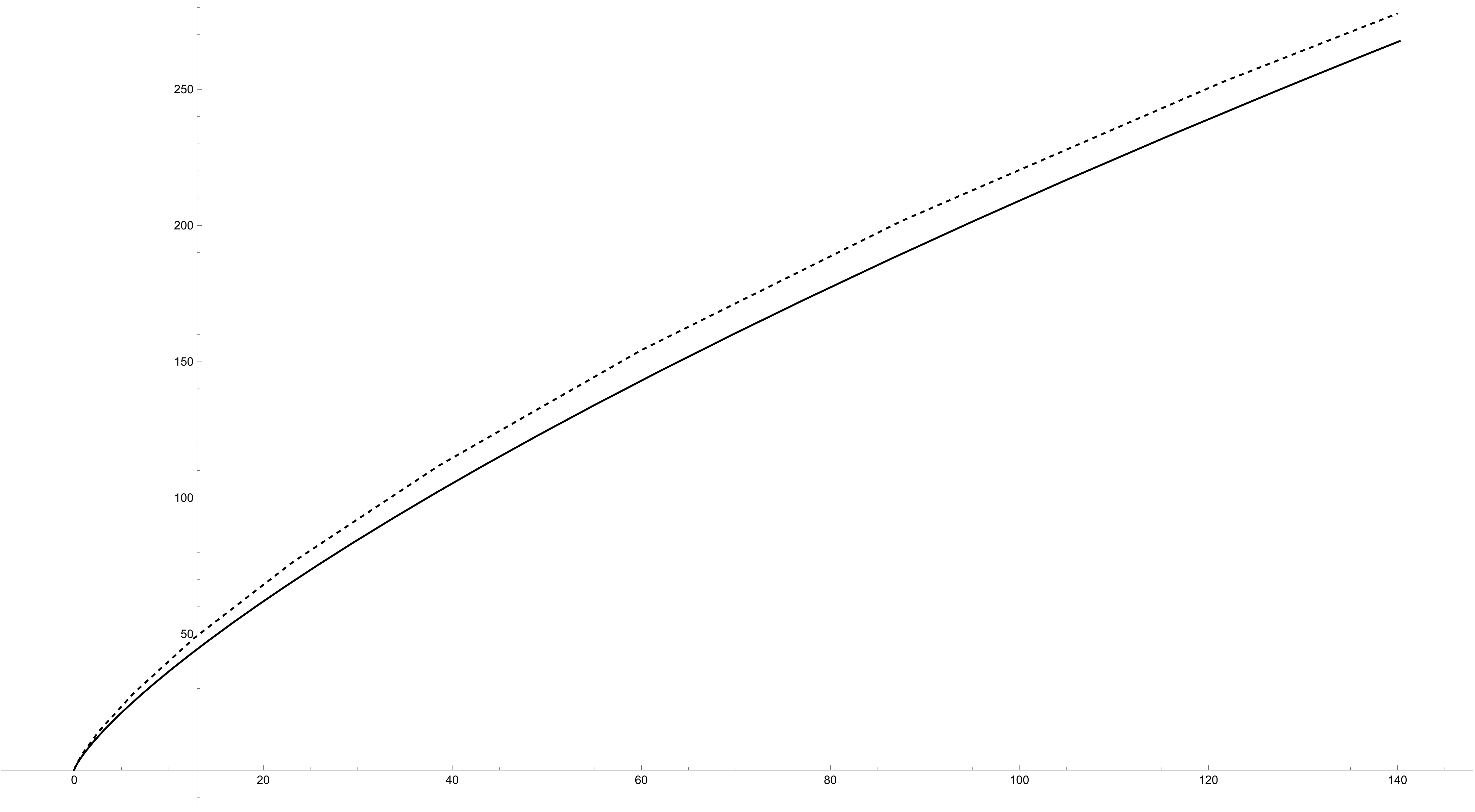}
	\caption{Comparison of   $ I_{(S^4\times \re, 2^{\frac{5}{3}} (g_0^4+dt^2))}(v)$ (dashed) and  $(0.867) I_{(S^5, 7.5 g_0^5)}(v)$ (continuous) for $13\leq v\leq 140 $.}
	\label{S513to140}
\end{figure}

\begin{figure}
	\includegraphics[scale=.2]{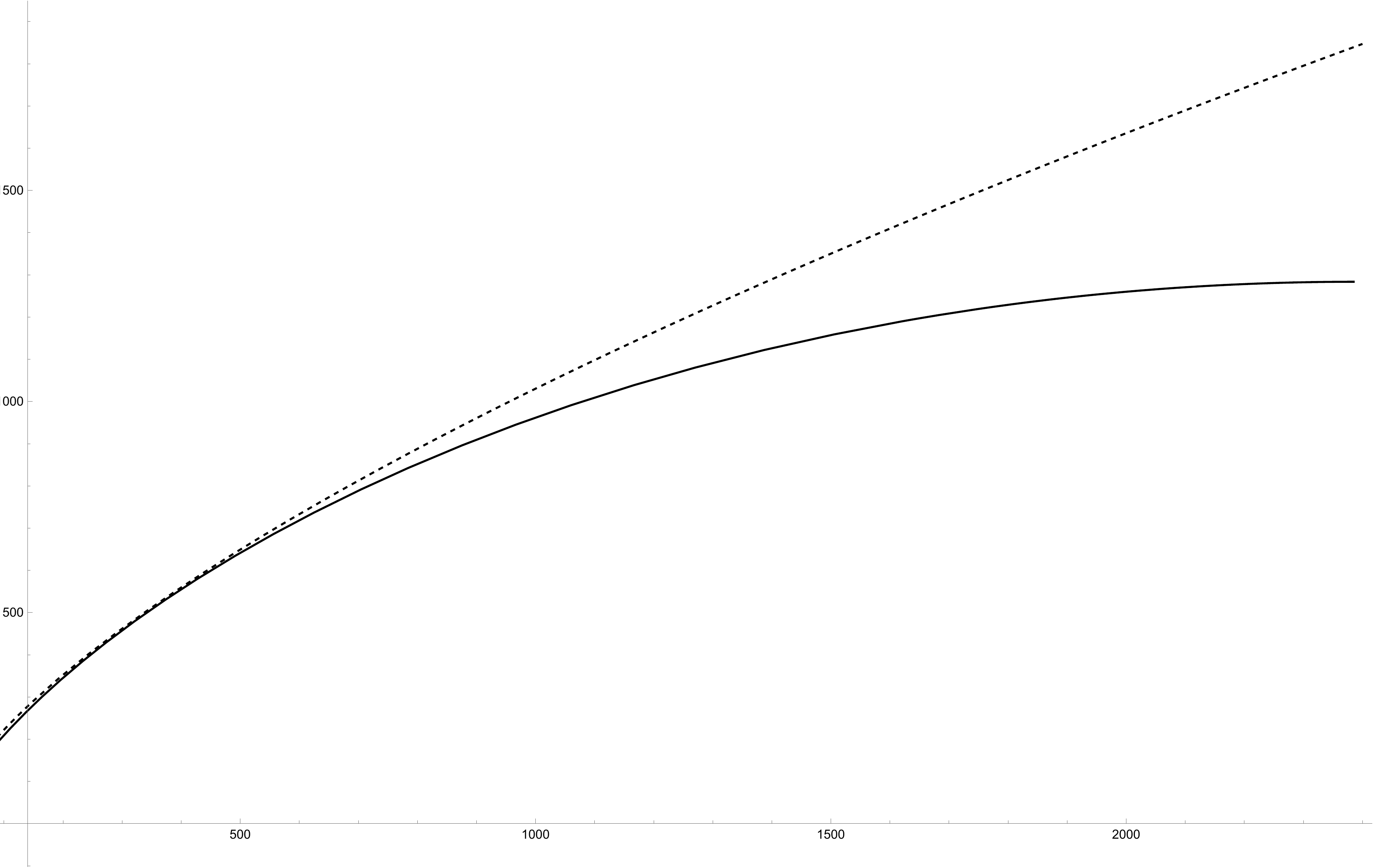}
	\caption{Comparison of   $ 10.3 v^{\frac{2}{3}}$ (dashed) and  $(0.867)I_{(S^5, 7.5 g_0^5)}(v)$ (continuous) for $140\leq v\leq 2389 $.}
	\label{S5140to2300}
\end{figure}

 \begin{Lemma}
	\label{23}
	For $v\geq0$, $ I_{(S^2\times \re^3,  (g_0^2+dt^2))}(v) \geq 0.867 I_{(S^5, 7.5 g_0^5)}(v)$.
	
\end{Lemma}

\begin{proof}
	\textbf{Case 1:}  For $0\leq v\leq13$.
	
	By direct computation $0.99 I_{(S^{4}\times \re, 2^{\frac{5}{3}}  (g_0^{4}+dt^2))}(13)> 13^{\frac{4}{5}} (6.34)$,	
	which we use as hypothesis of Lemma \ref{smalls}, to obtain: for $0\leq v\leq 13$,
	\begin{equation}
		\label{23case1}
		0.99 I_{(S^{4}\times \re, 2^{\frac{5}{3}}  (g_0^{4}+dt^2))}(v)>  \frac{ 6.34}{\gamma_5} I_{(S^{5},7.5 g_0^{5})}(v)>(0.867 )I_{(S^{5},7.5 g_0^{5})}(v)
	\end{equation}
	
	From this and lemma \ref{Is4r2}, we conclude case 1.
	
	\textbf{Case 2:}  For $13\leq v\leq140$.

Since the isoperimetric profiles $ I_{(S^{4}\times \re,2^{\frac{5}{3}}  (g_0^{4}+dt^2))}(v)$ and $I_{(S^{5},7.5 g_0^{5})}(v)$ are known, it follows from direct   computations (see figure \ref{S513to140}) that for $13\leq v\leq140$,
$$	0.99 I_{(S^{4}\times \re, 2^{\frac{5}{3}}  (g_0^{4}+dt^2))}(v)> (0.867 )I_{(S^{5},7.5 g_0^{5})}(v)$$
 
From this inequality and lemma \ref{Is4r2}, this case is satisfied.

\textbf{Case 3:} For $v \geq 140$.

By direct computation, $	0.99 I_{(S^{4}\times \re, 2^{\frac{3}{2}} (g_0^{4}+dt^2))}(140)>277.8$. Hence, by   lemma \ref{Is4r2}: $I_{(S^2\times \re^3,  (g_0^2+dt^2))}(140)>277.8$. We  then use $x_0=140$, $y_0=277.8$ as the hypothesis for Lemma \ref{x0y0} and get, for $v\geq 140$, $I_{(S^2\times \re^3,  (g_0^2+dt^2))}(v)>10.3 v^{\frac{2}{3}}.$

Now, from direct  computations (see figure \ref{S5140to2300}), we have	$10.3 v^{\frac{2}{3}}> (0.867)I_{(S^{5},7.5g_0^{5})}(v)$,  for $140\leq v \leq 2389$.	 Hence, we have,  for $140\leq v \leq 2389$,
	
	\begin{equation}
		\label{last23}
		I_{(S^2\times \re^3,  (g_0^2+dt^2))}(v)>(0.867)I_{(S^{5},7.5 g_0^{5})}(v).
	\end{equation}

	Finally, we note that $I_{(S^{5},2.77 g_0^{5})}(v)$ reaches its maximum at $v=\frac{Vol((S^{5},7.5 g_0^{5}))}{2}\approx 2388.21$, while $I_{(S^2\times \re^3,  (g_0^2+dt^2))}(v)$ is non decreasing, by Lemma \ref{x0y0}. This implies that inequality (\ref{last23}) is still valid for $v \geq 2389$.

\end{proof}

Lemmas \ref{22a},  \ref{22b},  \ref{32},  \ref{23} imply Theorem \ref{Iso}.

Corollary \ref{Yamabe} follows immediately from Theorem \ref{Iso} in this article and Theorem 1.1 in \cite{PetRu2}. We note that $I_{(M^m\times \re^n,g+dx^2)}$ is non-decreasing, by Lemma \ref{x0y0}.

\begin{Theorem} \textnormal{(Theorem 1.1 in \cite{PetRu2})}
	Let $(M^m, g)$ be a closed Riemannian manifold with scalar curvature $s_g\geq m(m-1)$. If $I_{(M^m\times \re^n,g+dx^2)}$ is a non-decreasing
	function and $I_{(M^m\times \re^n,g+dx^2)}\geq \lambda I_{(S^{n+m},\mu g_0^{m+n})}$, then
$$Y(S^m\times \re^n,[g+dx^2])\geq min\left\{\frac{\mu m(m-1)}{(m+n)(m+n-1)}, \lambda^2\right\}Y(S^{m+n})$$
\end{Theorem}

 We also recall Theorem 1.1 in \cite{AFP}:
 \begin{Theorem} \textnormal{(Theorem 1.1 in \cite{AFP})}
 	 Let $(M^m, g)$ be a closed Riemannian m-manifold $(m \geq 2)$
 	of positive scalar curvature and $(N^n,h)$ any Riemannian closed n-manifold $n \geq 2$:
 	Then,
 	
 	$$\lim_{r\rightarrow \infty}Y(M\times \re^n,[g+rh])=Y(M\times \re^n,[g+g_E])$$
 \end{Theorem}

Corollary \ref{Yamabe2}  follows immediately from Corollary \ref{Yamabe} and Theorem 1.1 in \cite{AFP}, by noting that 

$$Y(M\times N)\geq \sup_{\{r>0\}} Y(M\times N,[g+rh])\geq \lim_{r\rightarrow \infty}Y(M\times \re^n,[g+rh])=Y(M\times \re^n,[g+g_E]).$$

\begin{comment}

\end{comment}

\end{document}